\documentclass[11pt]{article}
\usepackage{fullpage,amsmath,amsthm,amssymb,bm,hyperref,cleveref,graphicx,url}
\usepackage[linesnumbered,ruled,vlined]{algorithm2e}
\usepackage{authblk} 

\graphicspath{{./Fig/}}

\newtheorem{theorem}{Theorem}
\newtheorem{lemma}[theorem]{Lemma}
\newtheorem{corollary}[theorem]{Corollary}

\newtheorem{definition}[theorem]{Definition}

\def\calX{\mathcal{X}}
\def\bbF{\mathbb{F}}
\def\bbR{\mathbb{R}}
\def\bbN{\mathbb{N}}

\def\TV{\mathrm{TV}}
\def\dmu{\mathrm{d}\mu}

\def\cl{\mathrm{cl}}
\def\dx{\mathrm{d}x}

\def\bbR{\mathbb{R}}
\def\CS{\mathrm{CS}}
\def\KL{\mathrm{KL}}
\def\TV{\mathrm{TV}}

\def\supp{\mathrm{supp}}

\def\calE{\mathcal{E}}
\def\calF{\mathcal{F}}
\def\calX{\mathcal{X}}

\def\eqdef{{:=}}

\def\tr{{\mathrm{tr}}}
\def\st{{\ :\ }}
\def\inner#1#2{{\langle #1,#2\rangle}}

\title{The statistical Minkowski distances:\\ Closed-form formula for Gaussian Mixture Models}

\author{Frank Nielsen}
\affil{Sony Computer Science Laboratories, Inc.}
\affil{Tokyo, Japan}
\affil{{\small\tt Frank.Nielsen@acm.org}}

\date{}

\begin{document}
\maketitle
\begin{abstract}
The traditional Minkowski distances are induced by the corresponding Minkowski norms in real-valued vector spaces.
In this work, we propose novel statistical symmetric distances based on the Minkowski's inequality for probability densities 
belonging to Lebesgue spaces.
These statistical Minkowski distances admit closed-form formula for Gaussian mixture models when parameterized by integer exponents.
This result extends to arbitrary mixtures of exponential families with natural parameter spaces being cones: 
This includes the binomial, the multinomial, the zero-centered Laplacian, the Gaussian and the Wishart mixtures, among others.
We also derive a Minkowski's diversity index of a normalized weighted set of probability distributions from Minkowski's inequality.
\end{abstract}

\noindent {\bf Keywords}: Minkowski $\ell_p$ metrics, $L_p$ spaces, Minkowski's inequality, statistical mixtures, exponential families, multinomial theorem, statistical divergence, information radius, projective distance, scale-invariant distance, homogeneous distance.

\section{Introduction and motivation}

\subsection{Statistical distances between mixtures}

Gaussian Mixture Models (GMMs) are flexible statistical models often used in machine learning, signal processing and computer vision~\cite{JRGMM-2009,L22GMM-2011} since they can arbitrarily closely approximate any smooth density.
To measure the dissimilarity between probability distributions, one often relies on the principled information-theoretic 
Kullback-Leibler (KL) divergence~\cite{IT-2012}, commonly called the relative entropy.
However the lack of closed-form formula for the KL divergence between GMMs\footnote{When the GMMs share the same components, it is known that the KL divergence between them amount to an equivalent Bregman divergence~\cite{wmixture-2018} that is however computationally intractable because  its corresponding Bregman generator is the differential negentropy that does not admit a closed-form expression in that case.} 
has motivated various KL lower and upper bounds~\cite{KLGMM-2003,GMMHungarian-2005,GMMLSE-2016,CROT-2018} for GMMs or approximation techniques~\cite{KLGMM-2012}, and further spurred the  {\em design} of novel distances that admit closed-form formula between GMMs~\cite{CFMixture-2012}.
To give a few examples, let us cite the statistical squared Euclidean distance~\cite{L22GMM-2011,tL2-2012}, the Jensen-R\'enyi divergence~\cite{JRGMM-2009} (for the quadratic R\'enyi entropy), the Cauchy-Schwarz (CS) divergence~\cite{CS-2006,CS-2011}, and a statistical distance based on discrete optimal transport~\cite{OTWeight-2000,CROT-2018}.

A {\em distance} $D: \calX\times \calX \rightarrow \bbR$ is a non-negative real-valued function $D$ on the {\em product space} $\calX\times \calX$ such that $D(p,q)=D((p,q))=0$ iff. $p=q$.
A distance $D(p:q)$ between  $p$ and $q$ may not be symmetric:
This fact is emphasized by the ':' delimiter notation: $D(p:q)\not = D(q:p)$.
For example, the KL divergence is an oriented distance:
$\KL(p:q)\not=\KL(q:p)$. 
Two usual symmetrizations of the KL divergence are the Jeffreys' divergence and the Jensen-Shannon divergence~\cite{symdiv-2010}.
Informally speaking, a {\em divergence}\footnote{Also called a contrast function in~\cite{Eguchi-1983}.} is a {\em smooth distance}\footnote{A Riemannian distance is not smooth but a squared Riemannian distance is smooth.} that allows one to define an information-geometric structure~\cite{IG-2016}. In other words, a divergence is a smooth premetric distance~\cite{Deza-2009}.

Recently, the Cauchy-Schwarz divergence~\cite{CS-2006} has been generalized to H\"older divergences~\cite{HolderDiv-2017}.
These Cauchy and H\"older distances $D(p:q)$ are said to be {\em projective} because $D(\lambda p:\lambda' q)=D(p:q)$ for any $\lambda,\lambda'>0$. An important family of projective divergences for robust statistical inference are the $\gamma$-divergences~\cite{gammadiv-2008,PEFdiv-2016}. 
Interestingly, those projective distances
 do not require to handle normalized probability densities but only need to consider {\em positive densities} instead (handy in applications).  
The H\"older projective divergences do not admit closed-form formula for GMMs, except for the very special case of the CS divergence.
The underlying reason is that the conjugate exponents $\frac{1}{\alpha}+\frac{1}{\beta}=1$ of H\"older divergences would need to be both integers.
This constraint yields $\alpha=\beta=1$, giving the special case of the CS divergence (the other integer exponent case is in the limit when $\alpha=0$ and $\beta=\infty$).

\subsection{Minkowski distances and Lebesgue spaces}
The renown Minkowski distances are norm-induced metrics~\cite{Deza-2009} measuring distances between $d$-dimensional vectors $x,y\in\bbR^d$ defined for $\alpha\geq 1$ by:
\begin{equation}\label{eq:md}
M_\alpha(x,y)\eqdef \|x-y\|_\alpha = \left(\sum_{i=1}^d |x_i-y_i|^\alpha  \right)^{\frac{1}{\alpha}},
\end{equation}
where the Minkowski norms are given by $\|x\|_\alpha=\left(\sum_{i=1}^d |x_i|^\alpha  \right)^{\frac{1}{\alpha}}$.
The Minkowski norms can be extended to countably infinite-dimensional $\ell_\alpha$ spaces of sequences (see~\cite{HilbertSpace-2014}, p. 68).

Let $(\calX,\calF)$ be a measurable space where $\calF$ denotes the $\sigma$-algebra of $\calX$, 
and let $\mu$ be a probability measure (with $\mu(\calX)=1$) with full support $\supp(\mu)=\calX$ (where $\supp(\mu)\eqdef \cl(\{F\in \calF \st \mu(F)>0\})$ and $\cl$ denotes the set closure). 
Let $\bbF$ be the set of all real-valued measurable functions defined on $\calX$.
We define the {\em Lebesgue space}~\cite{HilbertSpace-2014} $L_\alpha(\mu)$  for $\alpha\geq 1$ as follows:
\begin{equation}
L_\alpha(\mu)  \eqdef \left\{ f\in \bbF \st  \int_\calX |f(x)|^\alpha \dmu(x) <\infty \right\}.
\end{equation}

The Minkowski distance~\cite{Minkowski-1910} of Eq.~\ref{eq:md} can be generalized to probability densities belonging to Lebesgue $L_\alpha(\mu)$ spaces, to get the {\em statistical Minkowski distance} for $\alpha\geq 1$:
\begin{equation}\label{eq:pmd}
M_\alpha(p,q)\eqdef \left(\int_\calX |p(x)-q(x)|^\alpha \dmu(x)  \right)^{\frac{1}{\alpha}}.
\end{equation}
When $\alpha=1$, we recover twice the {\em Total Variation} (TV) metric:
\begin{equation}
\TV(p,q) \eqdef \frac{1}{2} \int |p(x)-q(x)| \dmu(x)=\frac{1}{2} \|p-q\|_{L_1(\mu)}=\frac{1}{2}M_1(p,q).
\end{equation} 
Notice that the statistical Minkowski distance does not admit closed-form formula in general because of the absolute value.
The total variation is related to the probability of error in Bayesian statistical hypothesis testing~\cite{BayesTV-2014}.

In this work, we design novel distances based on the Minkowski's inequality (triangle inequality for $L_\alpha(\mu)$, which proves that
$\|p\|_{L_\alpha(\mu)}$ is a norm (i.e., the $L_\alpha$-norm), so that the statistical Minkowski's distance between functions of a Lebesgue space can be written as
 $M_\alpha(p,q)=\|p-q\|_{L_\alpha(\mu)}$).
The space $L_\alpha(\mu)$ is a Banach space (ie., complete normed linear space).

\subsection{Paper outline}
The paper is organized as follows:
Section~\ref{sec:MIGD} defines the new Minkowski distances by measuring in various ways the tightness of the Minkowski's inequality applied to probability densities.
Section~\ref{sec:CF} proves that all these statistical Minkowski distances admit closed-form formula for mixture of exponential families with conic natural parameter spaces for integer exponents. 
In particular, this includes the case of Gaussian mixture models.
Section~\ref{sec:CEF} lists a few examples of common exponential families with conic natural parameter spaces.
In Section~\ref{sec:MPD}, we define Minkowski's diversity indices for a normalized weighted set of probability distributions.
Finally, section~\ref{sec:concl} concludes this work and hints at perspectives.

\section{Distances from the Minkowski's inequality\label{sec:MIGD}}

Let us state Minkowski's inequality:
\begin{theorem}[Minkowski's inequality]
For $p(x), q(x)\in L_\alpha(\mu)$ with $\alpha\in [1,\infty)$, we have the following Minkowski's inequality:
\begin{equation}
\left(\int |p(x)+q(x)|^\alpha \dmu(x)\right)^{\frac{1}{\alpha}} \leq 
\left(\int |p(x)|^\alpha \dmu(x)\right)^{\frac{1}{\alpha}} 
+
\left(\int |q(x)|^\alpha \dmu(x)\right)^{\frac{1}{\alpha}},
\end{equation}
with equality holding only when $q(x)=0$ (almost everywhere, a.e.), or when
 $p(x)=\lambda q(x)$ a.e. for $\lambda>0$ for $\alpha>1$.
\end{theorem}
The usual proof of Minkowski's inequality relies on H\"older's inequality~\cite{Proof-1964,HolderDiv-2017}.
Following~\cite{HolderDiv-2017}, we define distances by measuring in several ways the tightness of the Minkowski's inequality.
When clear from context, we shall write $\|\cdot\|_\alpha$ for short instead of $\|\cdot\|_{L_\alpha(\mu)}$.
Thus Minkowski's inequality writes as:
\begin{equation}
\|p+q\|_\alpha \leq \|p\|_\alpha+\|q\|_\alpha.
\end{equation}
Minkowski's inequality proves that the $L_\alpha$-spaces are normed vector spaces. 

Notice that when $p(x)$ and $q(x)$ are probability densities (i.e., $\int p(x)\dmu(x)=\int q(x)\dmu(x)=1$), Minkowski's inequality
becomes an equality iff. $p(x)=q(x)$ almost everywhere, for $\alpha>1$.
Thus we can define the following novel Minkowski's distances between probability densities satisfying the identity of indiscernibles:

\begin{definition}[Minkowski difference distance]
For probability densities $p,q\in L_\alpha(\mu)$, we define the Minkowski difference $D_\alpha(\cdot,\cdot)$ distance for $\alpha\in (1,\infty)$ as:
\begin{equation}
D_\alpha(p,q) \eqdef \|p\|_\alpha+\|q\|_\alpha - \|p+q\|_\alpha\geq 0.
\end{equation}
\end{definition}

\begin{definition}[Minkowski log-ratio distance]
For probability densities $p,q\in L_\alpha(\mu)$, we define the Minkowski log-ratio distance $L_\alpha(\cdot,\cdot)$ for $\alpha\in (1,\infty)$ as:
\begin{equation}
L_\alpha(p,q) \eqdef -\log \frac{\|p+q\|_\alpha}{\|p\|_\alpha+\|q\|_\alpha} = \log \frac{\|p\|_\alpha+\|q\|_\alpha}{\|p+q\|_\alpha} \geq 0.
\end{equation}
\end{definition}

By construction, all these Minkowski distances are symmetric distances: 
Namely, $M_\alpha(p,q)=M_\alpha(q,p)$, $D_\alpha(p,q)=D_\alpha(q,p)$ and $L_\alpha(p,q)=L_\alpha(q,p)$.

Notice that $L_\alpha(p,q)$ is {\em scale-invariant}\footnote{Like any distance based on the log ratio of triangle inequality gap induced by a homogeneous norm.}: $L_\alpha(\lambda p,\lambda q)=L_\alpha(p,q)$ for any $\lambda>0$.
Scale-invariance is a useful property in many signal processing applications.
For example, the scale-invariant Itakura-Saito divergence (a Bregman divergence) has been successfully used in Nonnegative Matrix Factorization~\cite{NMFIS-2009} (NMF).
Distance $D_\alpha(p,q)$ is {\em homogeneous} since $D_\alpha(\lambda p,\lambda q)=|\lambda|D_\alpha(p,q)$
 for any $\lambda\in\bbR$ (and so is distance $M_\alpha(p,q)$).

\section{Closed-form formula for statistical mixtures of exponential families\label{sec:CF}}

In this section, we shall prove that $D_\alpha$ and $L_\alpha$ between statistical mixtures are in closed-form for all integer exponents 
(and $M_\alpha$ for  all even exponents) for mixtures of exponential families with conic natural parameter spaces.

Let us first define the positively {\em weighed geometric integral} $I$ of a set $\{p_1,\ldots, p_k\}$ of $k$ probability densities of $L_\alpha(\mu)$ as:
\begin{equation}
I(p_1,\ldots, p_k;\alpha_1,\ldots,\alpha_k) \eqdef \int_\calX p_1(x)^{\alpha_1}\ldots p_k(x)^{\alpha_k} \dmu(x),\quad
\alpha\in\bbR_+^k.
\end{equation}

An {\em exponential family}~\cite{EF-1986,EF-2009} $\calE_{t,\mu}$ is a set  $\{p_\theta(x)\}_\theta$ of probability densities wrt. $\mu$ which densities can be expressed proportionally canonically as:
\begin{equation}
p_\theta(x) \propto \exp(t(x)^\top \theta),
\end{equation}
where $t(x)$ is a  $D$-dimensional vector of sufficient statistics~\cite{EF-1986}.
The term $t(x)^\top \theta$ can be written equivalently as $\inner{t(x)}{\theta}$, where $\inner{\cdot}{\cdot}$ denotes the scalar product on $\bbR^D$.
Thus the normalized probability densities of $\calE_{t,\mu}$ can be written as:
\begin{equation}
p_\theta(x) = \exp\left(t(x)^\top \theta -F(\theta)\right),
\end{equation}
where 
\begin{equation}
F(\theta) \eqdef \log\int_\calX \exp(t(x)^\top \theta) \dmu(x),
\end{equation}
is called the {\em log-partition function} (also called cumulant function~\cite{EF-1986} or log-normalizer~\cite{EF-2009}).
The natural parameter space is:
\begin{equation}
\Theta \eqdef \left\{  \theta\in\bbR^D \st \int_\calX \exp(t(x)^\top \theta) \dmu(x) <\infty \right\}.
\end{equation}
Many common distributions (Gaussians, Poisson, Beta, etc.) belong to exponential families in disguise~\cite{EF-1986,EF-2009}.

\begin{lemma}
For probability densities $p_{\theta_1},\ldots, p_{\theta_k}$ belonging to the same exponential family $\calE_{t,\mu}$, we have: 
\begin{equation}
I(p_{\theta_1},\ldots, p_{\theta_k};\alpha_1,\ldots,\alpha_k) = \exp\left(
F\left(\sum_{i=1}^k \alpha_i\theta_i\right) - \sum_{i=1}^k \alpha_i F(\theta_i) 
\right)<\infty,
\end{equation}
provided that $\sum_{i=1}^k \alpha_i\theta_i\in\Theta$. 
\end{lemma}

\begin{proof}
\begin{eqnarray*}
&&I(p_{\theta_1},\ldots, p_{\theta_k};\alpha_1,\ldots,\alpha_k) =
 \int \prod_{i=1}^k 
\left(\exp\left(\left( t(x)^\top \theta_i - F(\theta_i)\right)\right)\right)^{\alpha_i} \dmu(x),\\
&=& \int  \exp\left( t(x)^\top (\sum_i \alpha_i\theta_i) - \sum_i \alpha_i F(\theta_i) 
+ \underbrace{F\left(\sum_i\alpha_i\theta_i\right)-F\left(\sum_i\alpha_i\theta_i\right)}_{=0} \right) \dmu(x),\\
&=& \exp\left(F\left(\sum_i\alpha_i\theta_i\right)-\sum_i \alpha_i F(\theta_i)  \right) 
\underbrace{\int_\calX 
 \exp\left( t(x)^\top \left(\sum_i\alpha_i\theta_i\right)-F\left(\sum_i\alpha_i\theta_i\right)\right)\dmu(x)}_{=1},\\
&=& \exp\left(F(\sum_i\alpha_i\theta_i)-\sum_i \alpha_iF(\theta_i)  \right),
\end{eqnarray*}
since $\int_\calX  \exp\left( t(x)^\top (\sum_i\alpha_i\theta_i)-F(\sum_i\alpha_i\theta_i)\right)\dmu(x)=\int_\calX p_{\sum_i\alpha_i\theta_i}(x)\dmu(x)=1$, 
provided that  $\bar\theta\eqdef \sum_i \alpha_i\theta_i\in\Theta$ (and $p_{\bar\theta}\in\calE_{t,\mu}$).
\end{proof}

In particular, the condition $\sum_i \alpha_i\theta_i\in\Theta$ always holds when the natural parameter space $\Theta$ is a {\em cone}.
In the remainder, we shall call those exponential families with natural parameter space being a cone, {\em Conic Exponential Families} (CEFs) for short. Note that when $\sum_i \alpha_i\theta_i\not\in\Theta$, 
the integral $I(p_{\theta_1},\ldots, p_{\theta_k};\alpha_1,\ldots,\alpha_k)$ diverges
(that is, $I(p_{\theta_1},\ldots, p_{\theta_k};\alpha_1,\ldots,\alpha_k)=\infty$).

Observe that for a CEF density $p_\theta(x)$, we have $p_\theta(x)^\alpha$ in $L_\alpha(\mu)$ for {\em any} $\alpha\in [1,\infty)$.

\begin{corollary}\label{cor:I}
We have $I(p_{\theta_1},\ldots, p_{\theta_k};\alpha_1,\ldots,\alpha_k) = \exp\left(
F\left(\sum_i \alpha_i\theta_i\right) - \sum_i \alpha_iF(\theta_i)
\right)<\infty$ for probability densities belonging to the same exponential family with natural parameter space $\Theta$ being a cone.
\end{corollary}

We also note in passing that $I(p_1,\ldots, p_k;\alpha_1,\ldots,\alpha_k)<\infty$ for $\alpha\in\bbR^k$ for probability densities belonging to the same exponential family with natural parameter space being an {\em affine space} (e.g., Poisson or isotropic Gaussian families~\cite{fdivchi-2014}).

Let us define:
\begin{equation}\label{eq:Jdiv}
J_F({\theta_1},\ldots, {\theta_k};\alpha_1,\ldots,\alpha_k)\eqdef 
\sum_i \alpha_iF(\theta_i) - F\left(\sum_i \alpha_i\theta_i\right).
\end{equation} 
This quantity is called the {\em Jensen diversity}~\cite{BRJensen-2011} when $\alpha\in\Delta_k$ (the $(k-1)$-dimensional standard simplex), or Bregman information\footnote{Because $\sum_i \alpha_i B_F(\theta_i:\bar\theta)=J_F({\theta_1},\ldots, {\theta_k};\alpha_1,\ldots,\alpha_k)$ for the barycenter $\bar\theta=\sum_i\alpha_i \theta_i$, where $B_F(\theta:\theta')=F(\theta)-F(\theta')-(\theta-\theta')^\top \nabla F(\theta')$ is a Bregman divergence.} in~\cite{BD-2005}.
Although the Jensen diversity is non-negative when $\alpha\in\Delta_k$, this Jensen diversity of Eq.~\ref{eq:Jdiv} maybe negative when 
$\alpha\in\bbR_+^k$.
When $\alpha\in\bbR_+^k$, we thus call the Jensen diversity the {\em generalized Jensen diversity}.
It follows that we have:
\begin{equation}
I(p_{\theta_1},\ldots, p_{\theta_k};\alpha_1,\ldots,\alpha_k) = \exp\left(
-J_F({\theta_1},\ldots, {\theta_k};\alpha_1,\ldots,\alpha_k)
\right)
\end{equation}

The CEFs include the Gaussian family, the Wishart family, the Binomial/multinomial family, etc.~\cite{EF-1986,EF-2009,CFMixture-2012}.

Let us consider a finite  positive mixture $\tilde{m}(x)=\sum_{i=1}^k w_i p_i(x)$ of $k$ probability densities, where the 
weight vector $w\in\bbR_+^k$ are not necessarily normalized to one.

\begin{lemma}\label{lem:powerposmeas}
For a finite positive mixture $\tilde{m}(x)$ with components belonging to the same CEF, $\|\tilde{m}\|_{L_\alpha(\mu)}$ is {\em finite} and in {\em closed-form}, 
for any integer $\alpha\geq 2$.
\end{lemma}

\begin{proof}
Consider the multinomial expansion $\tilde{m}(x)^\alpha$  obtained by applying the multinomial theorem~\cite{MT-1968}:
\begin{equation}
\tilde{m}(x)^\alpha = \sum_{\substack{\sum_{i=1}^k \alpha_i=\alpha\\ \alpha_i\in\bbN}} 
\binom{\alpha}{\alpha_1,\ldots,\alpha_k} \prod_{j=1}^k (w_jp_j(x))^{\alpha_j},
\end{equation}
where
\begin{equation}
\binom{\alpha}{\alpha_1,\ldots,\alpha_k} \eqdef \frac{\alpha!}{\alpha_1! \times \ldots \times \alpha_k!} ,
\end{equation}
is the {\em multinomial coefficient}~\cite{IntroDM-2012}. 
It follows that:
\begin{equation}
\int \tilde{m}(x)^\alpha\dmu(x) = \sum_{\substack{\sum_i \alpha_i=\alpha\\ \alpha_i\in\bbN}} \binom{\alpha}{\alpha_1,\ldots,\alpha_k} 
\left(\prod_{j=1}^k w_j^{\alpha_j}\right) I(p_1,\ldots, p_k;\alpha_1,\ldots,\alpha_k).
\end{equation}

Thus the term $\int \tilde{m}(x)^\alpha\dmu(x)$ amounts to a positively weighted sum of integrals of monomials that are  positively weighted geometric means of mixture components. 
When $p_i=p_{\theta_i}$, since $I(p_{\theta_1},\ldots, p_{\theta_k};\alpha_1,\ldots,\alpha_k)<\infty$ using Eq.~\ref{cor:I}, we conclude that $\tilde{m}\in {L_\alpha(\mu)}$ for $\alpha\in\bbN$, and we get the formula:
\begin{equation}\label{eq:posmixk}
\|\tilde{m}\|_{L_\alpha(\mu)}=
\left(
\sum_{\substack{\sum_i \alpha_i=\alpha\\ \alpha_i\in\bbN}} \binom{\alpha}{\alpha_1,\ldots,\alpha_k} 
\left(
\prod_{j=1}^k w_j^{\alpha_j}
\right)
\exp\left(
-J_F({\theta_1},\ldots, {\theta_k};\alpha_1,\ldots,\alpha_k)
\right)
\right)^{\frac{1}{\alpha}}
, 
\end{equation}
for $\alpha\in\bbN$. 
\end{proof}

A naive multinomial expansion of $\tilde{m}(x)^\alpha$ yields $k^\alpha$ terms that can then be simplified.
Using the multinomial theorem, there are $\binom{k+\alpha-1}{\alpha}$ integral terms in the formula of $\int (\sum_{i=1}^k w_i p_i(x))^\alpha \dmu(x)$.
This number corresponds to the number of sequences of $k$ disjoint subsets whose union is $\{1,\ldots, \alpha\}$ (also called the number of ordered partitions but beware that some sets may be empty).

The multinomial expansion can be calculated efficiently using a generalization of Pascal's triangle, called {\em Pascal's simplex}~\cite{PascalSimplex-2012}, thus avoiding to compute from scratch all the multinomial coefficients.

We have the following generalized Pascal's recurrence formula for calculating the multinomial coefficients:
\begin{equation}
\binom{\alpha}{\alpha_1,\ldots,\alpha_k} = \sum_{i=1}^k \binom{\alpha-1}{\alpha_1,\ldots,\alpha_i-1,\ldots,\alpha_k},
\end{equation}
with the terminal cases
$\binom{\alpha}{\alpha_1,\ldots,\alpha_k} =0$ if there exists an $\alpha_i<0$.
Also by convention, we set conveniently $\binom{\alpha}{\alpha_1,\ldots,\alpha_k} =0$ if there exists $\alpha_i>\alpha$.

An efficient way to implement the multinomial expansion using nested iterative loops follows from this identity:
\begin{equation}
\left(\sum_{i=1}^k x_i\right)^\alpha = \sum_{\alpha_1=0}^\alpha \sum_{\alpha_2=0}^{\alpha_1} \ldots \sum_{\alpha_{k-1}=0}^{\alpha_{k-2}}
\binom{\alpha}{\alpha_1}  \binom{\alpha_1}{\alpha_2}\ldots \binom{\alpha_{k-1}}{\alpha_{k-2}}
x_1^{\alpha-\alpha_1}  x_2^{\alpha_1-\alpha_2} \ldots x_{k-1}^{\alpha_{k-2}-\alpha_{k-1}} x_k^{\alpha_{k-1}}.
\end{equation}

We are now ready to show when the statistical Minkowski's distances $M_\alpha, D_\alpha$ and $L_\alpha$ are in closed-form for mixtures of CEFs using Lemma~\ref{lem:powerposmeas}.

\begin{theorem}[Closed-form formula for Minkowski's distances]\label{thm:mcf}
For mixtures $m=\sum_{i=1}^k w_ip_{\theta_i}$ and $m'=\sum_{j=1}^{k'} w_j'p_{\theta_j'}$ of CEFs $\calE_{\mu,t}$,
$D_\alpha$ and $L_\alpha$ admits closed-form formula for integers $\alpha\geq 2$, and
 $M_\alpha$ is in closed-form when $\alpha\geq 2$ is an even positive integer. 
\end{theorem}

\begin{proof}
For $D_\alpha$ and $L_\alpha$, it is enough to show that $\|m\|_{L_\alpha(\mu)}, \|m'\|_{L_\alpha(\mu)}$
 and $\|m+m'\|_{L_\alpha(\mu)}$ are all in closed-form. This  follows from Lemma~\ref{lem:powerposmeas}
 by setting $\tilde{m}$ to be $m$, $m'$ and $m+m'$, respectively.
The overall number of generalized Jensen diversity terms in the formula of $D_\alpha$ or $L_\alpha$ is 
$O\left(\binom{k+k'+\alpha-1}{\alpha}\right)$.

Now, consider distance $M_\alpha$.
To get rid of the absolute value in $M_\alpha$ for even integers $\alpha$, we rewrite $M_\alpha$ as follows:
\begin{eqnarray*}
M_\alpha(m,m')&=&\|m-m'\|_{L_\alpha(\mu)}=\left(\int |m(x)-m'(x)|^\alpha \dmu(x)\right)^{\frac{1}{\alpha}},\\
&=& \left(\int 
\left(  \left(m(x)-m'(x)\right)^2  \right)^{\frac{\alpha}{2}}
\dmu(x)
\right)^{\frac{1}{\alpha}}.
\end{eqnarray*}

Let $\tilde{m}(x)=(m(x)-m'(x))^2$.
We have:
\begin{eqnarray}
\tilde{m}(x) &=& (m(x)-m'(x))^2,\\
&=& m(x)^2+m'(x)^2-2m(x)m'(x),\\
&=& \left(\sum_{i=1}^k w_ip_{\theta_i}(x)\right)^2 + \left(\sum_{j=1}^{k'} w_j'p_{\theta_j'}(x)\right)^2
-2\sum_{i=1}^k\sum_{j=1}^{k'} w_iw_j' p_{\theta_i}(x)p_{\theta_j'}(x).\label{eq:mm}
\end{eqnarray}
We have the density products $p_{\theta,\theta'}\eqdef p_{\theta}p_{\theta'}=I(p_{\theta},p_{\theta'};1,1)\in L_{\frac{\alpha}{2}}(\mu)$ (using Lemma~\ref{lem:powerposmeas}) for any $\theta,\theta'\in\Theta$ and $\alpha\geq 2$. 
When $\alpha=2$, $\frac{\alpha}{2}=1$, and we easily reach a closed-form formula for $M_2(m,m')$.
Otherwise, let us expand all the terms in Eq.~\ref{eq:mm}, and rewrite $\tilde{m}(x)=\sum_{l=1}^K w_l'' p_{\theta_l,\theta'_l}$.
Now, a key difference  is that $w_l''\in\bbR$, and not necessarily positive.
Nevertheless, since $\frac{\alpha}{2}\in\bbN$, we can still use the multinomial theorem 
to expand $\tilde{m}(x)^{\frac{\alpha}{2}}$, distribute the integral over all terms, and
compute elementary integrals $I(p_{\theta_1,\theta'_1},\ldots, p_{\theta_K,\theta'_K};\alpha_1',\ldots,\alpha_K')$ with $\sum_{l=1}^K \alpha_i'=\frac{\alpha}{2}$  in closed-form.
Thus $M_\alpha$ is available in closed-form for mixtures of CEFs for all even positive integers $\alpha\geq 2$.
The number of terms in the $M_\alpha$ formula is $O\left(\binom{\max(k^2,{k'}^2)+\alpha-1}{\alpha}\right)$.
\end{proof}

Note that there exists a generalization\footnote{There also exists a generalization of the multinomial theorem to real exponents, however, this is much less known in the literature (see~\url{http://fractional-calculus.com/multinomial_theorem.pdf}).} of the binomial theorem to {\em real} exponents $\alpha\in\bbR$ called {\em Newton's generalized binomial theorem} using an infinite series of general binomial coefficients:
\begin{equation}\label{eq:gbin}
(x_1+x_2)^\alpha=\sum_{i=0}^\infty \binom{\alpha}{i} x_1^{\alpha-i} x_2^i,
\end{equation}
with the  generalized binomial coefficient defined by:
$$
\binom{\alpha}{i} \eqdef \frac{\alpha (\alpha-1)\ldots (\alpha-i+1)}{i!}  = \frac{\Gamma(\alpha+1)}{\Gamma(\alpha-i+1)\Gamma(i+1)},
$$
where $\Gamma(x)\eqdef \int_0^\infty t^{x-1} e^{-t} \mathrm{d}t$ is the Gamma function extending the factorial: $\Gamma(n)=(n-1)!$.
Equation~\ref{eq:gbin} is only valid whenever the infinite series converge. That is, for $|x_1| \geq |x_2|$.
When extending to mixture densities (i.e., $(w_1p_1(x)+w_2p_2(x))^\alpha$) and taking the integral, we therefore need to split the integral  into two integrals depending on whether $w_1p_1(x)\geq w_2p_2(x)$, or not. 
Furthermore, we need to compute these integrals on truncated support domains: This becomes very tricky as the dimension of the support increase~\cite{Genz-2009}.

\section{Some examples of conic exponential families\label{sec:CEF}}

Let us report  a few conic exponential families with their respective canonical decompositions.
The measure $\mu$ is usually either the Lebesgue measure on the Euclidean space (i.e., $\dmu(x)=\dx$), or the counting measure.

\begin{itemize}

\item {\bf Bernoulli/multinomial families.} 
The Bernoulli density is $p(x; \lambda)=\lambda^x (1-\lambda)^{1-x}$  with $\lambda\in(0,1)=\Delta_1$, for $\calX=\{0,1\}$. 
The natural parameter is $\theta=\log\frac{\lambda}{1-\lambda}$ and the conic natural parameter space is
$\Theta=\mathbb{R}$. The log-partition function is $F(\theta)=\log (1+e^\theta)$. The sufficient statistics is $t(x)=x$.

The multinomial density generalizes the Bernoulli and the binomial densities.
Here, we consider the categorical distribution also called  ``multinoulli'' distribution.
The multinoulli density is given by:
$$
p(x;\lambda_1,\ldots,\lambda_d)= \prod_{i=1}^d \lambda_i^{x_i},
$$
where $\lambda\in\Delta_d$, the $(d-1)$-dimensional standard simplex.
We have $\calX=\{0,1\}^d$.
The sufficient statistic vector is $t(x)=(x_1,\ldots,x_{d-1})$.
The natural parameter is a $(d-1)$-dimensional vector with natural coordinates 
$\theta=\left(\log\frac{\lambda_1}{\lambda_d},\ldots,\log\frac{\lambda_{d-1}}{\lambda_d}\right)$.
The conic natural parameter space is  $\Theta=\bbR^{d-1}$ (ie., a non-pointed cone).
The log-partition function is $F(\theta)=\log (1+\sum_{i=1}^{d-1} e^{\theta_i})$.

\item {\bf Zero-centered Laplacian family.}
The density is $p(x;\sigma)=\frac{1}{2\sigma} e^{-\frac{|x|}{\sigma}}$ and the sufficient statistic is $t(x)=|x|$. 
The natural parameter is $\theta=-\frac{1}{\sigma}$ with the conic parameter space $\Theta=(-\infty,0)=\bbR_{--}$.
The log-normalizer is $F(\theta)=\log (\frac{2}{-\theta})$.
See~\cite{LaplacianMM-2007} for an application of Laplacian mixtures.

\item {\bf Multivariate Gaussian family.} 
The probability density of a $d$-variante Gaussian distribution is:
$$
p(x;\mu,\Sigma) =
 \frac{1}{ (2\pi)^{d/2} |\Sigma|^{1/2} } \exp \left( - \frac{(x-\mu)^T \Sigma^{-1}(x-\mu)}{2} \right),\quad x\in\bbR^d
$$
where $|\Sigma|$ denotes the determinant of the positive-definite matrix $\Sigma$.
The natural parameter consists in a vector part $\theta_v$ and a matrix part $\theta_M$:
$\theta=(\theta_v,\theta_M)=(\Sigma^{-1}\mu,\Sigma^{-1})$.
The conic natural parameter space is $\Theta=\mathbb{R}^d \times S^d_{++}$,
 where $S^d_{++}$ denotes the cone of positive definite matrices of dimension $d\times d$.
The sufficient statistics are $(x,xx^\top)$.
The log-partition function is: 
$$
F(\theta)=\frac{1}{2}\theta_v^T \theta_M^{-1}\theta_v-\frac{1}{2}\log |\theta_M|+\frac{d}{2}\log 2\pi.
$$

\item {\bf Wishart family}.
The probability density is
$$
p(X;n,S) = 
\frac{|X|^{\frac{n-d-1}{2}} e^{-\frac{1}{2}\tr(S^{-1}X)}}{2^{\frac{nd}{2}} |S|^{\frac{n}{2}}\Gamma_d\left(\frac{n}{2}\right)},
\quad X\in S_{++}^d
$$
with $S\succ 0$ denoting the scale matrix  and $n>d-1$ denoting the number of degrees of freedom,
where $\Gamma_d$ is the multivariate Gamma function: 
$$
\Gamma_d(x)=
\pi^{d(d-1)/4}\prod_{j=1}^d
\Gamma\left( x+(1-j)/2\right)
. 
$$
$\tr(X)$ denotes the trace of matrix $X$.
The natural parameter is composed of a scalar $\theta_s$ and a matrix part $\theta_M$:
 $\theta=(\theta_s,\theta_M)=(\frac{n-d-1}{2},S^{-1})$.
The conic natural parameter space is
$\Theta=\mathbb{R}_+\times S^d_{++}$.
The sufficient statistics are $(\log |X|,X)$.
The log-partition function is:
$$
F(\theta)=\frac{(2\theta_s+d+1)d}{2}\log 2 + \left(\theta_s+\frac{d+1}{2}\right)\log |\theta_M|+\log \Gamma_d\left(\theta_s+\frac{d+1}{2}\right).
$$
See~\cite{WishartMM-2011} for an application of Wishart mixtures.
 
\end{itemize}

\section{Minkowski's diversity index\label{sec:MPD}}

Informally speaking, a {\em diversity index} is a quantity that measures the variability of elements in a data set (i.e., the diversity of a population).
For example, the (sample) variance of a (finite) point set is a {diversity index}. 
Point sets uniformly filling a large volume have large variance (and a large diversity index) 
while point sets with points concentrating to their centers of mass have low variance (and a small diversity index).

Recall that the Jensen diversity index~\cite{Nielsen-2017} of a normalized weighted set $\{p_1=p_{\theta_1},\ldots, p_n=p_{\theta_n}\}$ of densities belonging to the same exponential family (also called information radius~\cite{IRad-1999} or Bregman information~\cite{BD-2005,BVP-2009}) is defined for a strictly convex generator $F$ by:
$$ 
J_F(\theta_1,\ldots, \theta_n;w_1,\ldots,w_n) \eqdef \sum_{i=1}^n w_i F(\theta_i)  - F\left(\sum_{i=1}^n w_i\theta_i\right)  \geq 0.
$$
When $F(\theta)=\frac{1}{2}\inner{\theta}{\theta}$, we recover from $J_F$ the variance.

We shall consider finite mixtures~\cite{McLachlan-2000,BD-2005} with {\em linearly independent} component densities.
Using Minkowski's inequality iteratively for $f_1,\ldots, f_n\in L_\alpha(\mu)$, we get:
\begin{equation}
\left(\int \left|\sum_{i=1}^n f_i(x)\right|^\alpha \dmu(x)\right)^{\frac{1}{\alpha}}
\leq
\sum_{i=1}^n \left(\int |f_i(x)|^\alpha\dmu(x)\ \right)^{\frac{1}{\alpha}}.
\end{equation}
When $\alpha>1$, equality holds when the $f_i$'s are proportional (a.e. $\mu$).
By setting $f_i=w_ip_i$, we define the {\em Minkowski's diversity index}:

\begin{definition}[Minkowski's diversity index]
Define the Minkowski diversity index of $n$ weighted probability densities of $L_\alpha(\mu)$ for $\alpha>1$ by:
\begin{eqnarray}\label{eq:mdivind}
J^M_\alpha(p_1,\ldots, p_n;w_1,\ldots, w_n)&\eqdef&
\sum_{i=1}^n w_i\left( \int p_i(x)^\alpha \dmu(x) \right)^{\frac{1}{\alpha}}
- \left(\int \left|\sum_i w_ip_i(x)\right|^\alpha \dmu(x) \right)^{\frac{1}{\alpha}},\\
&=& \sum_{i=1}^n w_i \|p_i\|_\alpha - \left\|\sum_{i=1}^n w_i p_i\right\|_\alpha \geq 0.
\end{eqnarray}
\end{definition}

It follows a closed-form formula for the Minkowski's diversity index of a weighted set of distributions (ie., a mixture) belonging to the same CEF:
\begin{corollary}
The Minkowski's diversity index of $n$ weighted probability distributions belonging to the same conic exponential family is 
finite and admits a closed-form formula for any integer $\alpha\geq 2$.
\end{corollary}

\section{Conclusion and perspectives\label{sec:concl}}
Designing novel statistical distances which admit closed-form formula for Gaussian mixture models is important for a wide range of applications in machine learning, computer vision and signal processing~\cite{CS-2006}.
In this paper, we proposed to use the Minkowski's inequality to design novel statistical symmetric Minkowski  distances by measuring the tightness of the inequality either as an arithmetic difference or as a log-ratio of the left-hand-side and right-hand-side of the inequality.
We showed that these novel statistical Minkowski distances yield closed-form formula for mixtures of exponential families with conic natural parameter spaces whenever the integer exponent $\alpha\geq 2$.
In particular, this result holds for Gaussian mixtures, Bernoulli mixtures, Wishart mixtures, etc.
We termed those families as Conic Exponential Families (CEFs).
We also reported a closed-form formula for the ordinary statistical Minkowski distance for even positive integer exponents.
Finally, we defined the Minkowski's diversity index of a weighted population of probability distributions (a mixture), and proved that this diversity index admits a closed-form formula when the distributions belong to the same CEF.

Let us conclude by listing the formula of the statistical Minkowski distances for $\alpha=2$ for comparison with 
the Cauchy-Schwarz (CS) divergence:

\begin{eqnarray*}
M_2(m_1,m_2) &\eqdef&  \|m_1-m_2\|_2,\\ 
D_2(m_1,m_2) &\eqdef&  \|m_1+m_2\|_2 - (\|m_1\|_2+\|m_2\|_2),\\
L_2(m_1,m_2) &\eqdef&  -\log \frac{\|m_1+m_2\|_2}{\|m_1\|_2+\|m_2\|_2},\\
\CS(m_1,m_2) &\eqdef&  -\log \frac{\|m_1m_2\|_1}{\|m_1\|_2\|m_2\|_2}= -\log \frac{\inner{m_1}{m_2}_2}{\|m_1\|_2\|m_2\|_2},
\end{eqnarray*}
where $\inner{f}{g}_2=\int f(x)g(x)\dmu(x)$ for $f,g\in L_2(\mu)$. 
Note that for $\alpha=2$, $L_2(\mu)$ is a Hilbert space when equipped with this inner product.
We get closed-form formula for these statistical Minkowski's distances between mixtures $m_1$ and $m_2$ of CEFs, as well as for the Cauchy-Schwarz divergence.
All those statistical distances can be computed in quadratic time in the number of mixture components.

Selecting a proper divergence from {\em a priori} first principles for a given application is a paramount but difficult task~\cite{Deza-2009}. 
Often one is left by checking experimentally the performances of a few candidate divergences in order to select the {\em a posteriori}
 `best' one.
We hope that these newly proposed statistical Minkowski's distances, $D_\alpha$ and scale-invariant $L_\alpha$, will prove experimentally useful in a number of applications ranging from computer vision to machine learning and signal processing. 

\vskip 0.3cm
Additional material is available from\\ 
\centerline{\url{https://franknielsen.github.io/MinkowskiStatDist/}}

\vskip 1cm
\noindent {Acknowledgments}: The author would like to thank Ga\"etan Hadjeres for his careful reading and feedback.

\bibliographystyle{plain}
\bibliography{MinkowskiInequalityGapDistanceBIB}

\end{document}